\documentclass[11pt]{amsart}

\usepackage{pb-diagram}

\usepackage{amsfonts}
\usepackage{amsmath}
\usepackage{amssymb}

\newcommand{\dbar}{\overline{\partial}}

\newcommand{\OO}{\mathcal{O}}

\newcommand{\ddbar}{\sqrt{-1}\partial\dbar}

\newtheorem{theorem}{Theorem}[section]
\newtheorem{proposition}{Proposition}[section]
\newtheorem{lemma}{Lemma}[section]
\newtheorem{conjecture}{Conjecture}[section]
\newtheorem{definition}{Definition}[section]
\newtheorem{corollary}{Corollary}[section]

\newcommand{\PP}{\mathbb{P}}
\newcommand{\CC}{\mathbb{C}}

\begin{document}

\title{On a conjecture of Candelas and de la Ossa}

\author{Jian Song}

\address{Department of Mathematics, Rutgers University, Piscataway, NJ 08854}

\email{jiansong@math.rutgers.edu}

\thanks{Research supported in
part by National Science Foundation grant DMS-0847524 and  a Sloan Foundation Fellowship.}

\begin{abstract} We prove that the metric completion of a canonical Ricci-flat K\"ahler metric on the nonsingular part of a projective Calabi-Yau variety $X$ with ordinary double point singularities, is a compact metric length space homeomorphic to the projective variety $X$ itself. As an application, we prove a conjecture of Candelas and de la Ossa for conifold flops and transitions.

\end{abstract}

\maketitle


\section{\bf Introduction}\label{section1}

Yau's solution to the Calabi conjecture in \cite{Ya1} gives the existence of a unique Ricci-flat K\"ahler metric in any given K\"ahler class. Calabi-Yau manifolds and Ricci-flat K\"ahler metrics play a central role in the study of string theory. A natural problem in both mathematics and physics is to understand how Calabi-Yau manifolds of distinct topological types can be connected via algebraic, analytic and geometric processes. In the algebraic aspect, this is exactly the well-known Reid's fantasy \cite{Re} built on deep works of Clemens \cite{Cl}, Friedman \cite{Fr}, Hirzebruch \cite{Hi} and many others (cf. \cite{T1, CaO, CGH, T2, Gro1, Gro2, GW2, Ro}). A geometric transition is an algebraic notion of connectedness for the moduli space of Calabi-Yau threefolds, which  involves with a birational contraction and a complex smooth deformation.  It can be considered as the three dimensional analogue of analytic deformations among $K3$ surfaces. A conifold transition is a special geoemtric transition, where the contracted variety has only ordinary double points as singularity. The first physical interpretation of a conifold transition is given by Strominger \cite{Str}. In geometric and analytic aspect, the geometric transition should be considered for Calabi-Yau varieties coupled with canonical metrics such as Ricci-flat K\"ahler metrics. In \cite{CaO}, Candelas and de la Ossa conjecture that an algebraic conifold transition should be also analytic and geometric, i.e., the transition should be continuous in suitable geometric sense (Conjecture \ref{conj}). The recent work of Rong and Zhang \cite{RZ} proves a version of their conjecture by showing  that an algebraic geometric transition is indeed continuous in Gromov-Hausdorff topology. The goal of the paper is to establish a strong version of Canelas and de la Ossa's conjecture for conifold transitions. In general, a geometric transition is not necessarily projective or even K\"ahler \cite{Fr, T2}. The balanced metrics on non-K\"ahler Calabi-Yau threefolds are proposed in the direction to study non-K\"ahler geometric transitions \cite{FLY}.

We now state the main results of the paper.  Let $f: X\rightarrow Y$ be a small contraction morphism of a smooth projective Calabi-Yau threefold $X$ such that $Y$ is a normal Calabi-Yau variety with only ordinary double points as singularities. Let $\mathcal{L}_0$ be an ample line bundle over $Y$ and $\alpha$ be a K\"ahler class on $X$. Then there also exists a unique Ricci-flat K\"ahler metric $g(t)\in c_1(\alpha+ t[\pi^*\mathcal{L}_0])$ for $t\in (0, 1]$ by Yau's theorem. There exists a unique singular Ricci-flat K\"ahler metric $g_Y$ associated to its K\"ahler current $\omega_Y\in c_1(\mathcal{L}_0)$ obtained in \cite{EGZ}. In particular, $\omega_Y$ has bounded local potentials on $Y$ and $g_Y$ is a smooth K\"ahler metric on $Y_{reg}$, the nonsingular part of $Y$ \cite{EGZ}.

\begin{theorem}\label{main1}  The metric completion of $(Y_{reg}, g_Y)$ is a compact length space homeomorphic to the projective variety $Y$ itself, denoted by $(Y, d_Y)$. Furthermore, $(X, g(t))$ converges to $( Y, d_Y)$ in Gromov-Hausdorff topology, as $t\rightarrow 0$.

\end{theorem}

Theorem \ref{main1} shows that the algebraic small contraction can be realized by a continuous deformation of smooth Calabi-Yau K\"ahler metrics in Gromov-Hausdorff topology. In fact,  much stronger estimates are obtained in section 4 for degeneration of the Calabi-Yau metrics near the exceptional rational curves of $f$. Theorem \ref{main1} can be also applied to conifold flops and transitions as stated in the following corollaries. We also remark that the convergence is in fact smooth outside the exceptional rational curves as shown in \cite{To1}.

\begin{corollary} \label{main2}  Let

\begin{equation}
\begin{diagram}\label{conflop}
\node{X} \arrow{se,b,}{f} \arrow[2]{e,t,..}{ }     \node[2]{X'} \arrow{sw,r}{f'} \\
\node[2]{Y}
\end{diagram}
\end{equation}
be a conifold flop between two smooth projective Calabi-Yau threefolds $X$ and $X'$. Let $(Y, d_Y)$ be the compact metric length space induced by the singular Ricci-flat K\"ahler metric $g_Y$ as in Theorem \ref{main1}. Then there exist a smooth family of smooth Ricci-flat K\"ahler metrics $g(t)$ of $X$ and a smooth family of smooth Ricci-flat K\"ahler metrics $g'(s)$ of $X'$ for $t, s \in (0, 1]$, such that $(X, g(t))$ and $(X', g'(s))$ converge to $(Y, d_Y)$ in Gromov-Hausdorff topology as $t, s\rightarrow 0$.

\end{corollary}

It shows that any discrete conifold flop between two Calabi-Yau threefolds can be connected by a continuous path of Calabi-Yau metrics in Gromov-Hausdorff topology.

Theorem \ref{main1} can also be applied to conifold transitions of Calabi-Yau threefolds. An algebraic geometric transition (cf. \cite{Ro, RZ}) is a triple $T(X, Y, Y_s)$ connecting Calabi-Yau threefolds of different topological types, where $Y$ is a singular Calabi-Yau variety obtained from $X$ by a birational contraction morphism and $Y_s$ is a smooth complex deformation of $Y$. A conifold transition  is a geometric transition such that  the contracted singular Calabi-Yau variety $Y$ has only ordinary double points as singularities. The precise definitions are given in Section 2. The following corollary shows that an algebraic conifold transition is also a diffeo-geometric transition via continuous families of Ricci-flat K\"ahler metrics.

\begin{corollary}\label{main3}
Let $T(X, Y, Y_s)$ be a conifold transition of projective Calabi-Yau threefolds
\begin{equation}
\begin{diagram}\label{geotran}
\node{X}    \arrow[2]{e,t}{f }     \node[2]{Y } \node[2]{Y_s} \arrow[2]{w}\\
\end{diagram}
\end{equation}
for $s\in \Delta$, where $\Delta$ is  the unit disc in $\CC$.
Then there exist a smooth family of smooth Ricci-flat K\"ahler metrics $g(t)$ of $X$ for $t \in (0, 1]$ and a smooth family of smooth Ricci-flat K\"ahler metrics $g_{Y_s}$ of $Y_s$ for $s\in \Delta^*$, such that $(X, g(t))$ and $(Y_s, g_{Y_s})$ converge to $(Y, d_Y)$ in Gromov-Hausdorff topology as $t, s\rightarrow 0$. Here $(Y, d_Y)$ is the compact length metric space given in Theorem \ref{main1}.

\end{corollary}
Corollary \ref{main3} proves a conjecture of Candelas and de la Ossa (Conjecture \ref{conj}) for conifold transitions by combining Theorem \ref{main1} and the results of Rong and Zhang \cite{RZ} (Theorem \ref{RZ}). If there exists an algebraic conifold transition between  Calabi-Yau threefolds of distinct topology, it can be also constructed as a continuous transition for algebraic Calabi-Yau varieties coupled with canonical Ricci-flat K\"ahler metrics in the Gromov-Hausdorff "moduli space". The convergence in Corollary \ref{main2} and Corollary \ref{main3} is stronger than Gromov-Hausdorff convergence and in fact, it is in local $C^\infty$-topology outside the exceptional locus as shown in \cite{To1}.

The organization of this article is as follows: In section 2, we give the background of conifold transitions, singular Ricci-flat K\"ahler metrics and complex Monge-Amp\`ere equations. In section 3, we review the Calabi symmetry and construct various local ansatz near an exceptional rational curve. In section 4, we obtain various uniform estimates for a degenerate family of Ricci-flat K\"ahler metrics. Finally, we prove the main results in section 5 with some discussion on generalizations to higher dimensional conifold transitions and to canonical surgery by the K\"ahler-Ricci flow.


\bigskip

\section{\bf Conifold transitions and complex Monge-Amp\`ere equations}\label{section2}

In this section, we give a brief introduction on conifold transitions and canonical Ricci-flat K\"ahler metrics on singular Calabi-Yau varieties..

\begin{definition} Let $X$ and $X'$ be two smooth Calabi-Yau threefolds with $f: X \rightarrow Y$ and $f' : X' \rightarrow Y$ being small contraction morphisms. Then the following diagram is called a flop between $X$ and $X'$

\begin{equation}
\begin{diagram}\label{flop}
\node{X} \arrow{se,b,}{f} \arrow[2]{e,t,..}{ }     \node[2]{X'} \arrow{sw,r}{f'} \\
\node[2]{Y}
\end{diagram}
\end{equation}

\end{definition}

$Y$ in (\ref{flop}) is a normal variety with canonical singularities and trivial canonical divisor. There  exists a unique Ricci-flat  K\"ahler metric in any given polarization by the work of Essydieux, Guedj and Zeriahi \cite{EGZ}.  This is shown by  solving degenerate complex Monge-Amp\`ere equations related to constant scalar curvature metrics (cf. \cite{EGZ, Z, Kol2, PhS, PhS2, PSS}) . Let $\mathcal{L}_0$ be an ample line bundle over $Y$ and so it induces an embedding morphism $Y \hookrightarrow \PP^N$  into some big projective space.  Let $\alpha$ be a K\"ahler class on $X$ and we define $\alpha_t= \alpha+ t [f^* \mathcal{L}_0]$. Obviously $\alpha_t$ is a K\"ahler class on $X$ whenever $t>0$.  Let $\theta\in [\mathcal{L}_0]$ be a multiple of the pullback of the Fubini-Study metric on $\PP^N$,  $\omega_0\in [\alpha]$ a K\"ahler metric on $X$ and $\omega_t = \theta + t\omega_0$.  Let $\Omega_{CY}$ be a smooth volume form on $X$ such that $\ddbar \log \Omega_{CY} =0$.
Then the solution of the following complex Monge-Amp\`ere equation
\begin{equation}\label{maeqn}
(\omega_t + \ddbar \varphi_t)^3 = c(t) \Omega_{CY},  ~~~  \sup_{X} \varphi_t = 0
\end{equation}
 gives rise to a Ricci flat K\"ahler metric $g(t)$ associated to the K\"ahler form
\begin{equation}
\omega(t) = \omega_t + \ddbar \varphi_t,
\end{equation}
where $c(t)$ is a family of constants in $t$ determined by $\alpha_t ^3 = c(t) \int_X \Omega_{CY}$.

The following deep estimates are obtained in \cite{EGZ, Z} built on techniques of Kolodziej \cite{Kol1}.
\begin{theorem}\label{EGZ1} There exists $C>0$ such that for all $t\in (0, 1]$,
\begin{equation}
||\varphi_t ||_{L^\infty(X)} \leq C.
\end{equation}

\end{theorem}

The local $C^\infty$ regularity follows from the $L^\infty$ estimate, Tsuji's trick \cite{Ts} and the general linear theory (cf. \cite{EGZ, ST1, To1}).
\begin{proposition}\label{EGZ2} Let $D$ be the exceptional locus of $f: X \rightarrow Y$ and $X^{\circ} = X\setminus D$. For any compact set $K$ of $X^{\circ} $ and $k \geq 0$, there exists $C_{K, k}>0$ such that for all $t\in (0, 1]$,
\begin{equation}
||\varphi_t ||_{C^k(K)} \leq C_{K, k}.
\end{equation}

\end{proposition}

From the above uniform estimates, we obtain a unique solution for the limiting degenerate complex Monge-Amp\`ere equation \cite{EGZ, To1}. It is shown in \cite{To1, ZY} that the diameter of $(X, g(t))$ is uniformly bounded for $t\in (0,1]$. The following corollary follows from Theorem \ref{EGZ1} and Proposition \ref{EGZ2} by letting $t\rightarrow 0$.

\begin{corollary}\label{EGZ3}  Let $Y_{reg}$ be the nonsingular part of $Y$. There exists a unique $\varphi_0 \in PSH(Y, \theta) \cap L^\infty(Y) \cap C^\infty (Y_{reg})$ such that $\sup_X \varphi_0 = 0$ and
\begin{equation}
(\theta + \ddbar \varphi_0)^3 = c_\theta \Omega_{CY}, ~~ \int_X \theta^3 = c_\theta \int_X \Omega_{CY}.
\end{equation}

\end{corollary}

The K\"ahler current
\begin{equation}
\omega_Y = \theta + \ddbar \varphi_0
\end{equation}
induces the unique singular Ricci-flat K\"ahler metric on $Y$ in $c_1(\mathcal{L}_0)$.

By the general theory in Riemannian geometry \cite{C, CC1, CC2, CCT}, one can always take the Gromov-Hausdorff limit for the family of $(X, g(t))$ with $t\in (0, 1]$. On the other hand, $g(t)$ converges to $g_Y$ in $C^\infty(X^\circ)$. Naturally, one would ask if the intrinsic limit of $(X, g(t))$ in Gromov-Hausdorff topology is homeomorphic to $Y$ as a projective variety,  and if it coincides with its extrinsic limit.

We now introduce the notion of geometric transitions for Calabi-Yau threefolds (cf. \cite{Ro, RuZ, RZ}).

\begin{definition} \label{defgt} Let $X$ be a Calabi-Yau threefold and $f: X \rightarrow Y$  be a contraction morphism from $X$ to a normal Calabi-Yau variety $Y$. Suppose that $Y$ admits a smooth projective deformation $\pi: \mathcal{M} \rightarrow \Delta$  over the unit disc $\Delta \in \CC$ such that $K_{\mathcal{M}/\Delta} = \OO_{\mathcal{M}}$ with smooth fibres $Y_s=\pi^{-1} (s)$  of Calabi-Yau three folds for $s \neq 0$ and $Y=Y_0$.

Then the following diagram is called a geometric transition $T(X, Y, Y_s)$ %
\begin{equation}
\begin{diagram}\label{geotran2}
\node{X}    \arrow[2]{e,t}{f }     \node[2]{Y } \node[2]{Y_s} \arrow[2]{w}\\
\end{diagram}.
\end{equation}
\end{definition}

\begin{definition}\label{defct} A geometric transition $T(X, Y, Y_s)$ is called a conifold transition if $Y$  admits only ordinary double points as singularity.

\end{definition}

\begin{definition}\label{defcflop} A flop between two Calabi-Yau threefolds $X$ and $X'$ as in (\ref{flop})is called a conifold flop if $Y$  admits only ordinary double points as singularity.

\end{definition}

A local model for conifold singularities is given by $$ \{ z\in \CC^4~|~z_1^2 + z_2^2 + z_3^2 + z_4^2 =0 \}.$$
A well-known example for a conifold transition is given in \cite{GMS}. Let $Y$ be the hypersurface in $\PP^4$ defined by $$z_3 g(z_0, ..., z_4)+ z_4 h(z_0, ..., z_4)=0$$ with generic homogeneous polynomials $g, h$ of degree $4$ in $[z_0, z_1, ..., z_4] \in \PP^4$. The singular locus of $Y$ is given by $\{ z_3=z_4=g(z)=h(z)=0\} $, which consists of $16$ ordinary double points. The small resolution of the singularities of $Y$ gives rise to a smooth Calabi-Yau threefold $X$ and  $Y$ can also be smoothed to generic smooth quintic threefolds in $\PP^4$.

Let $T(X, Y, Y_s)$ be a geometric  transition associated to a smoothing $\mathcal{M} \rightarrow \Delta$.  Let $\mathcal{L}$ be an ample line bundle over $\mathcal{M}$ and $\mathcal{L}_s = \mathcal{L}|_{Y_s}$ . Then there exists a unique smooth Calabi-Yau K\"ahler metric $g_{Y_s} \in c_1 (\mathcal{L}_s)$ for $s \in \Delta^*$. When $s=0$, there exists a unique singular Ricci-flat K\"aher metric  $g_Y$ by Theorem \ref{EGZ3} such that the associated  K\"ahler current $\omega_Y \in c_1(\mathcal{L}_0)$  has bounded local potential and $\omega_Y^3$ is a Calabi-Yau volume form on $Y$. In fact, $g_Y$ is smooth on $Y_{reg}$, the nonsingular part of $Y$.

Let $\alpha$ be a K\"ahler class of $X$. Then $\alpha_t = \alpha + t [\mathcal{L}_0]$ is a K\"ahler class of $X$ for $t\in (0, 1]$ and there exists a unique smooth Calabi-Yau K\"ahler metric $g(t) \in c_1 (\alpha_t)$.  The following is a natural mathematical formulation for a conjecture of Candelas and de la Ossa  \cite{RZ}.

\begin{conjecture}\label{conj} Let $T(X, Y, Y_s)$ be a conifold transition. The metric completion of $(Y_{reg}, g_Y)$ is a compact length metric space homeomorphic to $Y$ as a projective variety. If we denote such a metric space by $(Y, d_Y)$, then $(X, g(t))$ and $(Y_s, g_{Y_s} )$ converge to $(Y, d_Y)$ in Gromov-Hausdorff topology as $t, s \rightarrow 0$
\begin{equation}
\begin{diagram}\label{diagcand}
\node{(X, g(t) )}    \arrow[2]{e,t,b}{d_{GH} }    \node[2]{ (Y, d_Y)  } \node[2]{(Y_s, g_{Y_s} )} \arrow[2]{w,t}{d_{GH}}
\end{diagram}, ~~ ~ ~ t, s \rightarrow 0.
\end{equation}

\end{conjecture}

The following theorem of Rong and Zhang \cite{RZ} proves a version of the above conjecture for general  geometric transitions.
\begin{theorem}\label{RZ} Let $T(X, Y, Y_s)$ be a geometric transition. The metric completion of $(Y_{reg}, g_Y)$ is a compact length metric space and we denote it by $(Y', d_{Y'})$. Then $(X, g(t))$ and $(Y_s, g_{Y_s} )$ converge to $(Y', d_{Y'})$ in Gromov-Hausdorff topology as $t, s \rightarrow 0$.

\end{theorem}

The contribution of the paper is to give uniform estimates near the exceptional rational curves in the case of conifold transitions and prove the metric completion $(Y', d_{Y'})$ is homeomorphic to the projective variety $Y$ itself (cf. Theorem \ref{main1}).

The algebraic structure of conifold transitions are rather well understood (cf. \cite{Ro}). Let $T(X, Y, Y_s)$ be a conifold transition. $Y$ is then a normal Calabi-Yau threefold with isolated conifold singularities $y_1, y_2, ..., y_d$. If  $f: X \rightarrow Y$ is a minimal resolution of $Y$ at $y_1, y_2, ..., y_d$,  each component of the exceptional locus $D_j = f^{-1}(y_j)$ is a smooth rational curve $\PP^1$ with normal bundle $$N_{\PP^1} = \OO_{\PP^1} (-1) \oplus \OO_{\PP^1}(-1).$$ We define $X^{\circ} = X \setminus ( D_1\cup ... \cup D_d)$ and obviously $X^\circ$ is isomorphic to $Y_{reg}$. We will try to understand the local structure near these exceptional rational curves analytically in the next section.


\bigskip

\section{\bf Local  ansatz} \label{section3}

In this section, we will apply the Calabi ansatz introduced by Calabi \cite{C1} (also see \cite{Li, SY}) to understand the small contraction near the exceptional locus.

\bigskip

\noindent {\it  \large Calabi ansatz.} Let $E= \OO_{\PP^n}(-1)\oplus \OO_{\PP^n}(-1) \oplus ... \oplus \OO_{\PP^n}(-1)=\OO_{\PP^n}(-1)^{\oplus (n+1)}$ be the holomorphic bundle over $\PP^n$ of rank $n+1$.  Let $z=(z_1, z_2, ..., z_n)$ be a fixed set of inhomogeneous coordinates for $\PP^n$. and  $$\omega_{FS}= \ddbar \log (1+|z|^2) \in \OO_{\PP^n}(1) $$ be the Fubini-Study metric on $\PP^n$ and  $h$ be the hermitian metric on $\mathcal{O}_{\mathbb{P}^n} (-1)$ such that $Ric(h) = -\omega_{FS}$. This induces a hermtian metric $h_E$ on $E$ is given by
 $$h_E= h^{\oplus (n+1)}. $$
Under a local trivialization of $E$, we write
$$e^\rho = h_\xi (z) |\xi|^2, ~~~ \xi = (\xi_1, \xi_2, ..., \xi_{n+1}),$$ where $h_\xi(z)$ is a local representation for $h_E $ with
$h_\xi (z)= (1+|z|^2).$

Now we are going to  define a family of K\"ahler metrics on $E$ as below
\begin{equation}\label{metricrep1}
\omega = a \omega_{FS} + \ddbar u(\rho)
\end{equation}
for an appropriate choice of convex smooth function $u= u(\rho)$.  In fact, we have the following criterion due to Calabi \cite{C1} for the above form $\omega$ to be K\"ahler.
\begin{proposition}\label{kacon}

$\omega$ defined as above, extends to a global K\"ahler form on $E$ if and only if

\begin{enumerate}

\item[(a)]

$a>0$,

\item[(b)] $u'>0$ and $u''>0$ for $\rho\in (-\infty, \infty)$,

\item[(c)] $U_0 (e^\rho) = u(\rho)$ is smooth on $(-\infty, 0]$ with $U_0' (0)>0$.

\end{enumerate}

\end{proposition}

Straightforward calculations show  that
\begin{equation}\label{metricrep2}
\omega = (a + u'(\rho)) \omega_{FS} + h_\xi e^{-\rho} ( u' \delta_{\alpha \beta} + h_\xi e^{-\rho} ( u'' - u') \xi^{\bar \alpha} \xi^{\beta} ) \nabla \xi^\alpha \wedge \nabla \xi^{\bar \beta}.
\end{equation}
Here, $$\nabla \xi^\alpha = d \xi^\alpha + h_\xi^{-1} \partial h_\xi \xi^\alpha$$ and $\{ dz^i, \nabla \xi^\alpha\}_{i=1, ..., n, \alpha=1, 2, ..., n+1}$ is dual to the basis

\begin{equation*} \nabla_{z^i} = \frac{\partial}{\partial z^i} - h_\xi ^{-1} \frac{\partial h_\xi }{\partial z^i} \sum_\alpha \xi^\alpha \frac{\partial }{\partial \xi^\alpha}, ~~~~~~~ \frac{\partial}{\partial \xi ^\alpha}.
\end{equation*}

Let $L=\OO_{\PP^n}(-1)$. Then $E= L^{\oplus (n+1)}$. Let $p_\alpha: E \rightarrow L$ be the projection from $E$ to its $\alpha$th component, and let $e^{\rho_\alpha} = (1+|z|^2) |\xi_\alpha|^2$ and so $e^{\rho} = \sum_{\alpha=1}^{n+1} e^{\rho_\alpha}$.


\bigskip

\noindent{\it \large A local conifold flop.} From now on, we assume that $n=1$ and so $E$ is a rank two bundle over $\PP^1$.  Let $P_0$ be the zero section of $ E$ which is a rational curve with normal bundle $\OO_{\PP^1}(-1) \oplus \OO_{\PP^1}(-1)$. Then by contracting $P_0$, one obtains the variety $\hat E$ with only one isolated double point as singularity.  We can now define a flop for $E$ by letting $E'= \OO_{\PP^1} (-1)\oplus \OO_{\PP^1} (-1)\oplus \OO_{\PP^1}(-1)$. $\hat E$ is isomorphic to $E$ with a similar local trivialization $(w, \eta_1, \eta_2)$. Then the flop between $E$ and $E'$
\begin{equation}
\begin{diagram}\label{diag1}
\node{E} \arrow{se,b,}{f}  \arrow[2]{e,t,..}{\check{f} }     \node[2]{E'} \arrow{sw,r}{f'} \\
\node[2]{\hat E}
\end{diagram},
\end{equation} can be viewed as change of coordinates as below
$$ z = \frac{\eta_2}{\eta_1}, ~~\xi_1 = \eta_1, ~~\xi_2 = w\eta_1, $$
or
$$ w= \frac{\xi_2}{\xi_1}, ~~ \eta_1=\xi_1, ~~\eta_2 = z\xi_1.$$
We also have the following relation
$$e^\rho= (1+|z|^2) (|\xi_1|^2 + |\xi_2|^2) = (1+ |w|^2)( |\eta_1|^2 + |\eta_2|^2).$$

Let $\pi: E\rightarrow \PP^1$ and $\pi': E' \rightarrow \PP^1$. Then for each fixed $w\in \PP^1$ the proper transformation of $(\pi') ^{-1}(w)$ via $\check{f} ^{-1}$ is the hypersurface of $E$ given by
$$\xi_2 = w \xi_1.$$
Such a hypersurface is isomorphic to $\OO_{\PP^1}(-1)$ or simply $\CC^2$ blow-up at one point and we denote it by $L_w$. Hence we obtained a meromorphic family of isomorphic surfaces $L_w$ in $E$ parametrized by $w\in \PP^1 = \pi'(E')$. We  can also view $E$ as a meromorphic fibration of $\CC^2$ blow-up at one point over $\PP^1$.
The following lemma can be obtained by explicit calculations.
\begin{lemma} For each $w\in \PP^1$, $L_w$ is isomorphic to $\CC^2$ blow-up at one point. Furthermore, for $w_1\neq w_2$, $$L_{w_1} \cap L_{w_2} = P_0,$$
where $P_0$ is the zero section of $E$.

\end{lemma}

\bigskip

\noindent {\it \large Local forms.}  We will define two reference forms on $E$.  We first fix a K\"ahler form $\hat \omega$ on $E$ by
\begin{eqnarray*}
&&\hat\omega \\
&=& \omega_{FS} + \ddbar e^{\rho}\\
&= &(1+ (1+|z|^2)e^\rho) \omega_{FS} +  \sqrt{-1} z\bar \xi d\xi\wedge d\bar z + \sqrt{-1} \bar z \xi dz \wedge d\bar \xi +  \sqrt{-1} (1+|z|^2) d\xi \wedge d\bar \xi.
\end{eqnarray*}
Then we choose a smooth closed nonnegative real $(1,1)$ form $\tau$ defined by
\begin{eqnarray*}
 &&\tau\\
 &=& \ddbar e^{\rho_1}\\
 & =& \sqrt{-1} |\xi_1|^2 dz\wedge d\bar z +  \sqrt{-1}z\bar \xi_1 d\xi_1 \wedge d\bar z +  \sqrt{-1}\bar z \xi_1 dz \wedge d\bar \xi_1 +   \sqrt{-1}(1+|z|^2) d\xi_1 \wedge d\bar \xi_1\\
 & =& \sqrt{-1} (1+|z|^2) e^{\rho_1 }\omega_{FS} + \sqrt{-1} z\bar \xi_1 d\xi_1 \wedge d\bar z +  \sqrt{-1}\bar z \xi_1 dz \wedge d\bar \xi_1 +   \sqrt{-1}(1+|z|^2) d\xi_1 \wedge d\bar \xi_1.
 \end{eqnarray*}
Although $\tau$ is not big, it defines a flat degenerate K\"ahler form on $L_w$ for each $w$.

\begin{lemma}\label{locm1}  Let $\nu_1 = z \xi_1$, $\nu_2 = \xi_1$, and $\nu=(\nu_1, \nu_2)$.  Then $e^{\rho_1} = |\nu|^2$ and $$\tau =  \sqrt{-1} \left( d\nu_1 \wedge d\overline \nu_1 + d\nu_2\wedge d\overline \nu_2 \right)$$ is the pullback of the flat Euclidean metric on $\CC^2$. Hence  $\tau$ is flat on $L_w \setminus P_0$ for each $w\in \PP^1$.
\end{lemma}

The restriction of $\hat \omega$ on $L_w$ for fixed $w$ is given by
\begin{eqnarray*}
\hat \omega &=& \omega_{FS} + \ddbar e^\rho =\omega_{FS} +  (1+ |w|^2) \ddbar e^{\rho_1}\\
%
%
&=&  \left(1+ (1+ |w|^2 )(1+|z|^2)e^{\rho_1} \right) \omega_{FS}  \\
&&+  \sqrt{-1}  (1+ |w|^2 ) \left( z\bar \xi_1 d\xi_1 \wedge d\bar z + \bar z \xi_1 dz \wedge d\bar \xi_1 +  (1+|z|^2) d\xi_1 \wedge d\bar \xi_1 \right).
\end{eqnarray*}
We are only interested the local behavior of these forms near the zero section $P_0$, so we define
\begin{equation}\label{defome}\Omega= \{ \rho < 0\} \subset E,
\end{equation}
 or equivalently in local coordinates, $$e^\rho= (1+|z|^2)|\xi|^2= |\nu|^2 + (1+|\nu_1/\nu_2|^2)|\xi_2|^2\leq 1. $$
Then
\begin{equation}\label{compp}
L_w \cap \Omega= \{ (z, \xi_1, \xi_2)~|~ \xi_2= w \xi_1, ~e^{\rho_1} \leq (1+|w|^2)^{-1} \} .
\end{equation}
We now compare $\hat\omega$ and $\tau$ on each meromorphic fibre $L_w\simeq \OO_{\PP^1}(-1)$.

\begin{lemma} \label{loccom} For  all $w \in \CC$,
\begin{equation}\label{loccom2}
\tau  |_{L_w\cap \Omega } \leq  \hat \omega |_{L_w\cap \Omega } \leq 2 e^{-\rho_1} \tau |_{L_w\cap \Omega} .
\end{equation}
\end{lemma}

\begin{proof} The lower bound for $\hat\omega$ is trivial because $$\hat\omega\geq \ddbar e^\rho\geq\ddbar e^{\rho_1} = \tau.$$
Restricted on each $L_w \cap \Omega$, $$\ddbar e^\rho = (1+|w|^2)\ddbar e^{\rho_1} \leq e^{-\rho_1} \tau,$$ because $1+|w|^2 \leq e^{-\rho_1}$ by  (\ref{compp}). We also have on $L_w$, $$e^{\rho_1} \omega_{FS} \leq  \frac{  \sqrt{-1} |\nu_2|^2}{ 1+ |\nu_1/\nu_2|^2} d\left(\frac{\nu_1}{\nu_2} \right) \wedge d \overline{ \left(\frac{\nu_1}{\nu_2} \right)} \leq \ddbar |\nu|^2 = \tau.$$ The lemma  follows immediately as $$\hat\omega = \omega_{FS} + \ddbar e^{\rho} \leq 2e^{-\rho_1} \tau $$ restricted on $L_w\cap \Omega$.

\end{proof}
We also remark that  the estimate (\ref{loccom2}) still holds if one changes the trivialization by $U(2)$ action on $\xi=(\xi_1, \xi_2)$, because $\hat \omega$ is invariant by $U(2)$-action and the bounding coefficients do not depend on the choice of local trivialization as long as they are  equivalent by $U(2)$-action.


\bigskip
\noindent{\it \large A local model.}  We will now construct a family of complete Ricci-flat K\"ahler metrics on $E$. Such metrics are   given in \cite{CaO} and here we give the calculations in terms of  the Calabi ansatz.
Let  $\omega_E(t)=t \omega_{FS} + \ddbar u$ be a K\"ahler metric with Calabi symmetry defined on $E$ for $t\in (0, 1]$. Then the Ricci curvature of $\omega_E(t)$ is given by
$$ Ric(\omega_E(t) ) = -\ddbar \left(\log (t+u')u'u'' - 2\rho\right).$$
The vanishing Ricci curvature is equivalent the following equation
$$\left((t+u')u'u''\right)' = e^{2\rho},$$
and then by integration twice, we have
\begin{equation}\label{cubic}
2(u')^3 + 3 t (u')^2 - 3 e^{2\rho} = 0.
\end{equation}
 For each $t>0$, equation (\ref{cubic}) can be explicitly solved for $u'$ by the cubic formula and it is asymptotically of order $t^{-1/2} e^\rho$ near $\rho= -\infty$.

 when $t=0$, equation (\ref{cubic}) becomes $(u')^3 = 3e^{2\rho}/2$ and the solution is explicitly given by $$u_{\hat E}' = (3/2)^{1/3} e^{2\rho/3}, ~~u_{\hat E}''= (2/3)^{2/3} e^{2\rho/3}.$$  Such $u_{\hat E}$ induces   a complete Ricci-flat K\"ahler metric
 \begin{equation}\label{limcy}
 \omega_{CY, \hat E} = \ddbar u_{\hat E}
 \end{equation}
 on $\hat E$ with an isolated cone singularity.


\bigskip

\section{\bf Estimates}\label{section4}

From now on, we consider the small contraction morphism $$\pi: X \rightarrow Y $$ from a smooth Calabi-Yau threefold $X$ to a conifold $Y$. Without loss of generality, we assume that $y_1, ..., y_d$ are all the ordinary double points of $Y$ with $D_i = \pi^{-1} (y_j)$ for $j=1, ..., d$.

Due to the estimates in Proposition \ref{EGZ2} away from the exceptional rational curves $D_1$, ..., $D_d$, it suffices to prove a uniform estimate for the degenerating family of Calabi-Yau metrics in a small neighborhood of each exceptional rational curve. Without loss of generality, we localize the problem by looking at a neighborhood of a fixed irreducible rational curve $D$ isomorphic to $$\Omega= E\cap \{\rho <0\}$$ with $D =\{ \rho=-\infty\} $ as defined in (\ref{defome}).

Let $\omega(t)$ be the Ricci-flat K\"ahler metric on $X$ defined in (\ref{maeqn}) for $t\in (0, 1]$ with the same assumptions. We will restrict all the metrics  and apply estimates to  $\Omega$. Then for each $t\in (0, 1]$, $\omega(t)$ is equivalent to $\hat \omega$ on $\Omega$.

The goal in the section is to obtain a second order estimate for the local potential of $\omega(t)$. The usual method in \cite{Ya1, SW2} does not quite work in this case as there does not exist a good reference metric with admissible curvature properties, in particular, some component in the curvature tensor of $\hat\omega$ tends $-\infty$ near $P_0$.  The geometric interpretation of such difficulty is that the degenerate locus for the complex Monge-Amp\`ere equation (\ref{maeqn}) has codimension greater than one. Since $E$ admits a meromorphic family of $\CC^2$ blow-up at one point as shown in section 3, we consider a partial 2nd order estimate by bounding the metric along each meromorphic fibre. We therefore take advantage of the geometric flop structure and apply the maximum principle by a meromorphic slicing, so that the exceptional locus restricted to each meromorphic fibre has codimension one and we can apply ideas in \cite{SW2}. More precisely,  for any K\"ahler form $\omega$ on $E$, we can take the fibre-wise trace of $\omega$ with respect to $\tau$ along each $L_w$.

\begin{definition} We define for $t\in (0, 1]$,

$$H (t, \cdot)= tr_{\tau|_{L_w\cap \Omega}} (\omega (t) |_{L_w\cap \Omega}).$$

\end{definition}

Here $\tau$ and $\omega(t)$ are restricted to $L_w$ as smooth real closed $(1,1)$-forms. $H$ can also be expressed as
$$ H(t, \cdot)  = \frac{ \omega(t) \wedge \tau \wedge dw\wedge d\bar w}{\tau^2 \wedge dw \wedge d\bar w}. $$
\begin{lemma}\label{H1} $H \in C^\infty( \overline{\Omega} \setminus S)$  for all $t\in (0, 1]$, where $S=\{ \rho_1=-\infty \}$. Furthermore,

\begin{enumerate}

\item[(a)] for all $t\in (0, 1]$, $$\sup_{\Omega} e^{\rho_1} H(t, \cdot) < \infty ;$$

\item[(b)] there exists $C>0$ such that for  all $t\in (0, 1]$,

$$ \sup_{\partial \Omega} e^{\rho_1} H (t, \cdot)  \leq C.$$

\end{enumerate}

\end{lemma}

\begin{proof}  For each fixed $t\in (0,1]$, $\omega(t)$ is equivalent to $\hat\omega$ on $\Omega$ and so $(a)$ follows immediately from Lemma \ref{loccom}. By Proposition \ref{EGZ2},  there exists $C>0$ such that for all $t\in (0,1]$ and $p \in \partial \Omega$, $$\omega(t) \leq C \hat\omega.$$ Therefore for all $t\in (0, 1]$,
$$\sup_{\partial \Omega} e^{\rho_1} H(t, \cdot) \leq C \sup_{\partial \Omega} e^{\rho_1} tr_{\tau|_{L_w\cap\Omega}} (\hat\omega|_{L_w \cap \Omega}) \leq 2C$$ and it proves $(b)$.

\end{proof}

\begin{proposition}\label{H2} Let $\Delta_t$ be the Laplace operator associated to the Ricci-flat K\"ahler metric $g(t)$ for $t\in (0, 1]$. Then

\begin{equation}
\Delta_t \log H \geq 0.
\end{equation}

\end{proposition}

\begin{proof} We define $$I = \log H$$ and break the proof into the following steps.

\bigskip

\noindent {\it \large  Step 1.} We first make a choice of special coordinates. On $\Omega$, we have the standard local coordinates with Calabi symmetry as  defined in the previous section, i.e., for each $p \in \Omega$,  we have at $p$, $ (z(p), \xi_1(p), \xi_2(p))$.  Once we fix $p$, there exists a unique $w\in \PP^1$ such that $p\in L_w$.

\begin{enumerate}


\item[(a)] Near $p \in \Omega$, we first choose the coordinates $(\nu_1, \nu_2, w)$, where $\nu_1= \xi_1$ and $\nu_2=z\xi_1$ as defined in Lemma \ref{locm1} . We will apply a linear transformation to $(\nu_1, \nu_2, w)$ such that  $$x=(x_1, x_2, x_3)^T= A^{-1} (\nu_1, \nu_2, w)^T.$$   We assume that $A$ is in the form of
    $$\left(
        \begin{array}{cc}
           A' & a \\
          0 & 1 \\
        \end{array}
      \right), $$ where $A'$ is a $2\times 2$ matrix and $a$ is a $2\times 1$ vector. Immediately, we have $$x_3 = w.$$

\item[(b)]  Suppose $g(t)$ at $(t, p)$ is given by the following hermitian matrix with respect to coordinates $(\nu_1, \nu_2, w)$

$$ G= \left(
        \begin{array}{cc}
          B & b \\
          \overline{b}^T & c \\
        \end{array}
      \right),$$ where $B$ is a $2\times 2$ hermitian matrix, $b$ a $2\times 1$ vector.   Then under the new coordinates $x$, $g(t)$ at $p$ is given by the following hermitian matrix

    \begin{eqnarray*}
    \bar A^T G A &=& \left(
                      \begin{array}{cc}
                        \overline{A'}^T  B  A' & \overline{A'}^T B a + \overline{A'}^T b \\
                        \overline{a}^T B A' + \overline{b}^T A' &  \overline{a}^T B a + \overline{a}^T b + \overline{b}^T a + c  \\
                      \end{array}
                    \right) \\
    &=& \left(
                      \begin{array}{cc}
                        \overline{A'}^T  B  A' & \overline{A'}^T (B a + b) \\
                        (\overline{a}^T B + \overline{b}^T) A' &  \overline{a}^T B a + 2Re(\overline{a}^T b )+ c  \\
                      \end{array}
                    \right).
    \end{eqnarray*}

\item[(c)] We choose a unitary matix $A'$ such that $\overline{A'}^T B A'$ is diagonalized, i.e., $$ \left(
             \begin{array}{cc}
               \lambda_1 & 0 \\
               0 & \lambda_2 \\
             \end{array}
           \right) $$
     and choose $a$ such that $$ Ba= -b$$ since $B$ has rank $2$. Therefore under the coordinates $x$, at $(t, p)$,

$$g= \left(
      \begin{array}{ccc}
        \lambda_1 & 0 & 0 \\
        0 & \lambda_2 & 0 \\
        0  &  0 & \lambda_3  \\
      \end{array}
    \right),
$$
where $\lambda_3= c - \overline{a}^TB a.$
The matrix representation of $\tau$ under the coordinates $(\nu_1, \nu_2, w)$ is given by  $$\left(
              \begin{array}{ccc}
                1 & 0 & 0 \\
                0 & 1 & 0 \\
                0 & 0 & 0 \\
               \end{array} \right), $$
and so  its matrix representation under the coordinates $X$ at $(t, p)$ is given by              $$\tau= \overline{A}^T \left(
              \begin{array}{ccc}
                1 & 0 & 0 \\
                0 & 1 & 0 \\
                0 & 0 & 0 \\
              \end{array}
            \right) A= \left(
                         \begin{array}{ccc}
                           I_{2\times 2} & \overline{A'}^T a  \\
                           \overline{a}^T A' &  \overline{a}^T a \\
                         \end{array}
                       \right)
$$ since $A'$ is unitary. Since $x_3=w$ and on $L_w$, $x_3$ is constant and $\omega|_{L_w\cap \Omega} =\sqrt{-1} \sum_{i, j=1}^2 g_{i\bar j} dx^i \wedge d\overline{x^j}$.  Then at $(t, p)$, $$g|_{L_w} = \left(
                            \begin{array}{cc}
                              \lambda_1 & 0 \\
                              0 & \lambda_2 \\
                            \end{array}
                          \right), ~~~~~~~\tau |_{L_w} = \left(
                            \begin{array}{cc}
                              1 & 0 \\
                              0 & 1 \\
                            \end{array}
                          \right).
$$
Finally, we arrive at
$$ I(t, p) = \log \sum_{i, j=1, 2}(\tau |_{L_w}) ^{i\bar j} (g|_{L_w})_{i\bar j} = \log (\lambda_1 + \lambda_2).$$

 \end{enumerate}

\noindent {\it \large Step 2.}   Now we calculate $\Delta_t I$ at $(t, p)$ under the coordinates $x$.  Notice that $\tau|_{L_w}$ is a constant form $\sqrt{-1} (dx_1\wedge d \overline x_1 + dx_2 \wedge d\overline x_2) $ and so  all derivatives of $\tau|_{L_w}$ vanish.

We now apply the Laplace operator $\Delta_t$ to  $H$.
\begin{eqnarray*}
&&\Delta_t H \\
&=& \sum_{k, l=1}^3g^{k\bar l} \left( \sum_{i, j=12} \left( \tau|_{L_w} \right)^{i\bar j} g_{i\bar j} \right)_{k\bar l}\\
&=&\sum_{k, l=1}^3 \sum_{i, j=1, 2} g^{k\bar l} \left( \tau|_{L_w} \right)^{i\bar j} g_{i\bar j, k\bar l} - \sum_{k, l=1}^3\sum_{i, j, p, q=1, 2}g^{k \bar l}g_{i\bar j} \left( \tau|_{L_w} \right)^{i\bar q} \left( \tau|_{L_w} \right)^{ p\bar j} \tau_{p\bar q, k\bar l} \\
&=& - \sum_{k, l=1}^3 \sum_{i, j=1,2} g^{k\bar l} \left( \tau|_{L_w} \right)^{i\bar j}R_{i\bar j k\bar l} + \sum_{k, l, p, q=1}^3\sum_{i, j =1,2}g^{k\bar l} \left( \tau|_{L_w} \right)^{i\bar j} g^{p\bar q} g_{p \bar j, \bar l} g_{i\bar q, k}\\
&=& - \sum_{i, j=1, 2} \left( \tau|_{L_w} \right)^{i\bar j} R_{i\bar j} + \sum_{k, l, p,q=1}^3\sum_{i, j=1, 2}g^{k\bar l} \left( \tau|_{L_w} \right)^{i\bar j} g^{p\bar q} g_{p \bar j, \bar l} g_{i\bar q, k}\\
&=& \sum_{k,l, p, q=1}^3 \sum_{i, j=1, 2}g^{k\bar l} \left( \tau|_{L_w} \right)^{i\bar j}g^{p\bar q} g_{p \bar l, \bar j} g_{k\bar q, i}.
\end{eqnarray*}
Then
\begin{eqnarray*}
&& \Delta_t  I \\
&=&( H )^{-1}  \sum_{k, l=1}^3 \sum_{i, j=1, 2; p, q=1,2,3}g^{k\bar l}\left( \tau|_{L_w} \right)^{i\bar j} g^{p\bar q} g_{p \bar j, \bar l} g_{i\bar q, k} - (H )^{-2} |\nabla I |^2_g\\
&=& (H )^{-2}\sum_{k, l=1}^3\left( H   \sum_{i, j=1, 2; p, q=1,2,3}g^{k\bar l} \left( \tau|_{L_w} \right)^{i\bar j} g^{p\bar q} g_{p \bar j, \bar l} g_{i\bar q, k} -   g^{k\bar l} (\sum_{i, j=1, 2} \left( \tau|_{L_w} \right)^{i\bar j} g_{i\bar j, k})  (\sum_{i, j=1,2}\left( \tau|_{L_w} \right)^{j\bar i} g_{j\bar i, \bar l})    \right)\\
&=& (H )^{-2} \sum_{k, l=1}^3 \left( H   \sum_{i, j=1, 2; p, q=1,2,3}g^{k\bar l} \left( \tau|_{L_w} \right)^{i\bar j}g^{p\bar q} g_{p \bar l, \bar j} g_{k\bar q, i} -   g^{k\bar l} (\sum_{i, j=1, 2} \left( \tau|_{L_w} \right)^{i\bar j} g_{i\bar j, k})  (\sum_{i, j=1, 2} \left( \tau|_{L_w} \right)^{j\bar i} g_{j\bar i, \bar l})    \right).
\end{eqnarray*}

\noindent {\it \large Step 3.}   The proof of the proposition is now reduced to show that
$$\sum_{k, l=1}^3 \left(H  \sum_{i, j=1, 2; p, q=1,2,3}g^{k\bar l} \left( \tau|_{L_w} \right)^{i\bar j} g^{p\bar q} g_{p \bar l, \bar j} g_{k\bar q, i} -   g^{k\bar l} (\sum_{i, j=1, 2} \left( \tau|_{L_w} \right)^{i\bar j} g_{i\bar j, k})  (\sum_{i, j=1, 2} \left( \tau|_{L_w} \right)^{j\bar i} g_{j\bar i, \bar l})    \right)\geq 0. $$
Note that $\tau_{i\bar j}=\delta_{ij}$ for $i, j=1, 2$ and  $g= (\lambda_1, \lambda_2, \lambda_3)$. Then
\begin{eqnarray*}
&&\sum_{k, l=1}^3 g^{k\bar l} \left(\sum_{i, j=1, 2} \left( \tau|_{L_w} \right)^{i\bar j} g_{i\bar j, k} \right)  \left(\sum_{i, j=1,2} \left( \tau|_{L_w} \right)^{j\bar i} g_{j\bar i, \bar l} \right)\\
&=&\sum_{k=1, 2,3} \lambda_k^{-1}  |\sum_{i=1, 2}g_{i\bar i, k}|^2\\
&\leq& \sum_{i, j=1,2} \left( \sum_{k=1, 2,3} \lambda_k^{-1} |g_{i\bar i, k}|^2\right)^{1/2}  \left(\sum_{k=1, 2,3} \lambda_k^{-1} |g_{j \bar j, k}|^2 \right)^{1/2}\\
&=&\left( \sum_{i=1,2} \left(\sum_{k=1, 2,3} \lambda_k^{-1} |g_{i\bar i, k}|^2 \right)^{1/2}  \right)^2\\
&=&\left(   \sum_{i=1,2} \lambda_i^{1/2} \left( \sum_{k=1, 2,3} \lambda_k^{-1}\lambda_i^{-1} |g_{i\bar i, k}|^2 \right)^{1/2}               \right)^2\\
&\leq& \left(\sum_{i=1, 2} \lambda_i \right) \left( \sum_{i=1, 2;k=1, 2,3} \lambda_k^{-1}\lambda_i^{-1} |g_{i\bar i, k}|^2 \right) \\
&\leq& H  \left( \sum_{k, l=1, 2, 3; i=1, 2} \lambda_k^{-1}\lambda_l^{-1} |g_{i\bar l,  k}|^2 \right) \\
&=& H   \sum_{i, j=1,2; k, l, p,q=1,2, 3}  g^{k\bar l} \left( \tau|_{L_w} \right)^{i\bar j}  g^{p\bar q} g_{p\bar j, \bar l} g_{i\bar q, k}  .
\end{eqnarray*}

This completes the proof of the proposition.

\end{proof}

\begin{corollary} There exists $C>0$ such that on $\Omega$, for  all $t\in (0, 1]$,
\begin{equation}
  H \leq C e^{-\rho_1}.
 \end{equation}

\end{corollary}

\begin{proof}  Let $$I_{\epsilon} = \log H + (1+\epsilon) \rho_1$$ for $\epsilon>0$. Let $S= \{\rho_1=-\infty\} $. Then  for all $\epsilon>0$, $\limsup_{p \rightarrow S} I_\epsilon = -\infty$ by Lemma \ref{H1},  and on $\Omega\setminus S$, $$\Delta_t I_\epsilon >0,$$ because of Proposition \ref{H2} and the fact that $\Delta_t \rho_1= \Delta_t \log (1+|z|^2)|\xi_1|^2 = tr_{\omega}(\omega_{FS})>0$ on $\Omega\setminus S$.

Applying the maximum principle for $I_\epsilon$, we know that the maximum of $I_\epsilon$ has to be achieved on $\partial \Omega$.  Then by Lemma \ref{H1}, there exists $C>0$ such that for all $\epsilon\in (0,1]$ and $t\in(0,1]$,
$$\sup_{\Omega\setminus S} I_\epsilon = \sup_{\partial \Omega} I_\epsilon \leq \sup_{\partial \Omega} I_0  \leq C.$$ The corollary is then proved by letting $\epsilon \rightarrow 0$.

\end{proof}

We define a holomorphic vector $V$ on $\Omega$ by
$$ V=   \xi_1 \frac{\partial}{\partial \xi_1} +  \xi_2 \frac{\partial}{\partial \xi_2}.$$ $V$ vanishes along $P_0$ and $$|V|^2_{\hat \omega} = e^\rho.$$ We also consider the normalized vector field
$$W= \frac{V}{|V|_{\hat \omega}} = e^{-\rho/2} \sum_{\alpha=1,2} \xi_\alpha \frac{\partial}{\partial \xi_\alpha}.$$

Now we can obtain uniform bounds for the degenerating Ricci-flat K\"ahler metrics $g(t)$ near the exceptional curves.

\begin{proposition} \label{keyest} There exists $C>0$ such that for all $t\in (0, 1]$  and on $\Omega$,
\begin{equation} \label{tanest}
C^{-1}\omega_{\hat E} \leq \omega(t) \leq Ce^{-\rho} \omega_{\hat E},
\end{equation}
and
\begin{equation}\label{verest}
|W|^2_{g(t)} \leq Ce^{-\rho/2},
\end{equation}
where $\omega_{\hat E}= \ddbar e^\rho$.

\end{proposition}

\begin{proof}  We break the proof into the following steps.

\bigskip

\noindent {\it \large Step 1.}   We apply the similar argument in the proof of Schwarz lemma \cite{Y2, ST1}. Notice that $\omega (t) = \omega_t + \ddbar \varphi$ with $\varphi\in C^\infty(X)$ uniformly bounded in $L^\infty(X)$ for $t\in (0, 1]$. Also there exists $C_1>0$ such that for all $t\in (0,1]$ and on $\Omega$, $\omega_t\geq   C_1 \omega_{\hat E}$ on $\Omega.$ Then we consider the quantity $$L = \log tr_{\omega} (\omega_{\hat E}) -  \varphi.$$
$\omega_{\hat E}$ restricted to $\Omega$ is in fact the pullback of a flat metric on $\CC^4$ given by a local morphism $(\xi_1, \xi_2, z\xi_1, z\xi_2)$. Then straightforward calculations give
$$\Delta_t L \geq  tr_{\omega}(\omega_t) - 3 \geq  tr_\omega(\omega_{\hat E}) - 3.$$
Applying the maximum principle, we have
$$tr_\omega(\omega_{\hat E}) \leq \sup_{\partial \Omega} tr_{\omega}(\omega_{\hat E}) +3.$$
Note that $tr_{\omega(t)}(\omega_{\hat E}) $ is uniformly bounded on $\partial \Omega$. Hence $tr_{\omega}(\omega_{\hat E})$ is uniformly bounded above and so there exists $C_1>0$ such that
\begin{equation}\label{sch1}
\omega\geq C_1 \omega_{\hat E}.
\end{equation}

\bigskip

\noindent {\it \large Step 2.}   Since $\omega^3$ is uniformly equivalent to $\hat \omega^3$ and $e^{-\rho} \omega_{\hat E}^3$  in $\Omega$, there exists $C_2>0$ such that
\begin{equation}\label{sch2}
 \omega^3 \leq C_2 e^{-\rho} \omega_{\hat E}^3.
\end{equation}
 By the estimates (\ref{sch1}) and (\ref{sch2}), there exists $C_3>0$ such that
$$tr_{\omega_{\hat E}} (\omega) \leq  C_3 e^{-\rho} $$ and so $\omega\leq C_3 e^{-\rho} \omega_{\hat E}.$
This completes the proof for estimate (\ref{tanest}).

\bigskip

\noindent {\it \large Step 3.}   Let $V_1 = \xi_1 \frac{\partial}{\partial \xi_1}$ be the holomorphic vector field on $\Omega$. Then $V_1$ vanishes on $\rho_1=-\infty$ and
$ |V_1|^2_{\hat \omega} = (1+|z|^2) |\xi_1|^2= e^{\rho_1}.$
Using the normal coordinates for $\omega$, we can show that

$$\Delta_t |V_1|^2_g = - (V_1)^i (V_1)^{\bar j} R_{i\bar j} + g^{k\bar l} g_{i\bar j} \left((V_1)^{i})_k ((V_1)^{\bar j} \right)_{\bar l} = |\partial V|^2_g$$
and so
$$\Delta_t \log |V_1|^2_g = (|V_1|^2_g)^{-2} \left( |V_1|^2_g |\partial V_1|^2_g - |\nabla_t |V_1|_g^2|^2  \right)\geq 0. $$

We now define
$$G_\epsilon = \log \left( e^{\epsilon \rho_1}|V_1|^2_\omega tr_{\tau|_{L_w}} (\omega|_{L_w}) \right) = I + \log |V_1|_\omega^2 + \epsilon \rho_1$$
for $\epsilon\in (0, 1]$.
For each $t\in (0, 1]$, $G_\epsilon$ is smooth in $\Omega$ away from $ \rho_1 = -\infty $,  and it tends to $-\infty$ near $\rho_1=-\infty $  for all $\epsilon \in (0,1]$ by Lemma \ref{H1}. Furthermore, there exists $C_4>0$ such that for all $\epsilon\in (0,1]$ and $t\in (0, 1]$,
$$ \sup_{\partial \Omega} G_\epsilon  \leq C_4. $$
On the other hand,
$$ \Delta_t G_\epsilon= \Delta_t I + \Delta_t \log |V_1|^2_\omega + \epsilon \Delta_t \rho_1> 0.$$
By the maximum principle, $$\sup_{\Omega} G_\epsilon \leq \sup_{\partial \Omega } G_\epsilon \leq C_4.$$
Then by letting $\epsilon$ tend to $0$, we have
\begin{equation}\label{dirc}
|V_1|^2_\omega tr_{\tau|_{L_w}} (\omega|_{L_w}) \leq C_4.
\end{equation}

\bigskip

\noindent {\it \large Step 4.}   We will apply the estimate (\ref{dirc}) to prove (\ref{verest}).  Under the coordinates $(w, \nu_1, \nu_2)$, we have
$$\left( \tau|_{L_w} \right)_{\nu_i \bar\nu_j }= \delta_{ij}, ~~~V_1 = \nu_1\frac{\partial }{\partial \nu_1} + \nu_2 \frac{\partial}{\partial \nu_2} - w\frac{\partial }{\partial w}, ~~~|V_1|_{\tau}^2 = |\nu|^2= e^{\rho_1}, $$
At any $p\in L_w$ with $w=0$, we have
$$V_1 = V, ~~\rho_1 = \rho, $$
\begin{eqnarray*}
|V_1|^2_g &=&  |\nu_1|^2  g_{\nu_1 \bar \nu_1} + 2Re (\nu_1 \bar \nu_2 g_{\nu_1\bar\nu_2}) + |\nu_2|^2  g_{\nu_2 \bar \nu_2} \\
& \leq &  (|\nu_1|^2 + |\nu_2|^2) \left(  g_{\nu_1 \bar \nu_1} + g_{\nu_2 \bar \nu_2 }  \right) \\
&=& e^{\rho_1} tr_{\tau|_{L_0}} (\omega|_{L_0})  .
\end{eqnarray*}
Combined with (\ref{dirc}),   there exists $C_4>0$ such that for all $t\in (0, 1]$ and on $L_0 \cap \Omega$,
$$|V_1|_\omega^2  \leq C_4 e^{ \rho_1/2}, ~ or ~ |V|_\omega^2  \leq C_4 e^{ \rho/2}.$$
Equivalently, we have, $$ \left( |W|^2_\omega  \right) |_{L_0\cap \Omega} \leq C_7 e^{-\rho/2}|_{L_0\cap \Omega}.$$
Notice that $W$, $V$, $\rho$ are $U(2)$-invariant in terms of $\xi$ and all the bounds we have derived do not depend on the choice of trivialization differing by $U(2)$-action. Therefore we have
$$ |W|^2_\omega \leq C_7 e^{-\rho/2}$$ uniformly for $t\in (0, 1]$ and $\Omega$. This completes the proof of the proposition.

\end{proof}

The uniform bound on $diam(X, g(t))$ is already known to the general Calabi-Yau degeneration due to \cite{To1}. The following corollary follows from Proposition \ref{keyest}.

\begin{corollary} There exists $C>0$ such that for all $t\in (0, 1]$,

\begin{equation}
diam(X, g(t)) \leq C, ~~~~diam(X\setminus \{ D_1, ..., D_d\}, g(t)) \leq C.
\end{equation}

\end{corollary}

\begin{proof} Since $|W|_{g(t)}^2 \leq Ce^{-\rho/2}$, any point $p=(z, \xi)$ in $\Omega$ can be connected by a radial path $\gamma_{p}$ defined by $$(z, s \xi), ~ s \in \left[0, \frac{1}{(1+|z|^2)|\xi|^2} \right] $$ to $P_0$ and $\partial \Omega$ with $p' = \gamma \left(\frac{1}{(1+|z|^2)|\xi|^2} \right)\in \partial \Omega$. Then
the arc length of $\gamma_{p}$ with respect to $g(t)$ is uniformly bounded, i.e., there exists $C'>0$ such that for all $t\in (0, 1]$ and for all $p\in \Omega$, $$|\gamma_p |_{g(t)} \leq C'.$$
On the other hand, $g(t)$ is uniformly equivalent to $\hat \omega$ on $\partial \Omega$. Given any two points $p, q\in \Omega$, we can joint $p, q$ by $\gamma_p$, $\gamma_q$ and a smooth geodesic path $\gamma_{p', q'}$ with respect to $\hat \omega$ joining $p'$ and $q'$ in $\partial \Omega$.
Therefore,  both $diam(\Omega, g(t))$ and $diam(\Omega\setminus P_0, g(t))$ are  uniformly bounded and this completes the proof of the corollary.

\end{proof}

The following corollary shows that the restriction of $g(t)$ to the exceptional rational curve is uniformly bounded above.
\begin{corollary}
There exists $C>0$ such that for all $t\in (0, 1]$ ,

\begin{equation}
\omega(t)|_{P_0} \leq C \omega_{FS}|_{P_0}.
\end{equation}

\end{corollary}

\begin{proof}

By (\ref{tanest}) in Proposition \ref{keyest}, there exist $C_1, C_2>0$ such that for all $t\in (0, 1]$ and on $\Omega$,
$$\frac{\omega\wedge \omega_{\hat E}^2}{\omega_{FS}\wedge \omega_{\hat E}^2} \leq C_1e^{-\rho}  \frac{\omega_{\hat E}^3}{\omega_{FS}\wedge \omega_{\hat E}^2} \leq C_2.$$
For any point $p\in P_0$, there exist $e_1 \in T_p P_0$ and $e_2, e_3 \in T_p E$ such that they form an orthonormal basis of $T_p E$ with respect to $\hat \omega$. Obviously, $$\omega_{\hat E} (e_1, \cdot ) = 0.$$
Then $$tr_{\omega_{FS} |_{P_0}} ( \omega|_{P_0}) = \frac{\omega (e_1\wedge \overline{e_1})}{\omega_{FS} (e_1\wedge \overline{e_1})}    =\frac{ \omega\wedge\omega_{\hat E}^2 (e_1\wedge \overline{e_1}\wedge \cdot\cdot\cdot \wedge e_3\wedge \overline{e_3} ) } { \omega_{FS}\wedge \omega_{\hat E}^2 (e_1\wedge \overline{e_1}\wedge \cdot\cdot\cdot \wedge e_3\wedge \overline{e_3} ) } \leq C_2. $$

\end{proof}

In fact, the following proposition shows that exceptional rational curve become extinct as $t\rightarrow 0$.
\begin{proposition} \label{diam2} There exists $C>0$ such that for all $t\in (0, 1]$ such that
\begin{equation}
diam(P_0, g(t)|_{P_0}) \leq C t^{1/3}.
\end{equation}

\begin{proof} There exists $C>0$ such that for all $t\in (0, 1]$,
$$ \int_{P_0} \omega(t) = P_0 \cdot [\alpha_t] = t P_0 \cdot [\alpha] \leq C t. $$
Then the proposition is proved by the same argument in the proof of Lemma 3.2 in \cite{SW2}.

\end{proof}

\end{proposition}

We define for $r>0$,
\begin{equation}
\Omega_{ r } = \{ (z, \xi) \in E~|~ e^{\rho} = (1+|z|^2)|\xi|^2 \leq r^2 \} .
\end{equation}

Then we have the following proposition.

\begin{proposition}
 For any $\epsilon>0$, there exist $\varepsilon>0$ and  $\delta>0$ such that for all $t\in (0, \varepsilon)$,
\begin{equation}
diam(\Omega_\delta, g(t)) < \epsilon.
\end{equation}

\end{proposition}

\begin{proof} Given any two points $p, q \in \Omega_\delta$, there exist $p', q'\in P_0$ such that $p$ and $q$ can be connected to $p'$ and $q'$ by radial paths $\gamma_{p, p'}$ and $\gamma_{q, q'} $. For any $\epsilon>0$, we choose $\delta >0 $ such that the arc length of $\gamma_{p, p'}$ and $\gamma_{q, q'} $  is smaller than $\epsilon/3$ with respect to $g(t)$ for all $t\in (0, 1]$ by applying (\ref{verest}) in Proposition \ref{keyest}. By Proposition \ref{diam2}, we can choose $\varepsilon>0$ sufficiently small such that for $t\in (0, \varepsilon)$,
$$diam(P_0, g(t)|_{P_0}) \leq \frac{\epsilon}{3}.$$
Then
$$ dist_{g(t)} (p, q) \leq dist_{g(t)} (p, p') + dist_{g(t)} (q, q') + diam(P_0, g(t)|_{P_0}) < \epsilon $$
and the proposition follows.

\end{proof}


\bigskip

\section{\bf Proof of Theorem \ref{main1} and its generalizations}\label{section5}

Let $T(X, Y, Y_s)$ be a conifold transition. Let $g_Y$ be the unique singular Calabi-Yau K\"ahler metric associated to the K\"ahler current on $Y$ as defined in section 2.  Note that  $g_Y$ is smooth in $Y_{reg}=Y\setminus \{ y_1, ..., y_d \} $ and so we define a similar distance function on $Y$ as in Definition 5.1 in \cite{SW2}.

\begin{definition}  \label{defndy}  We extend  $g_Y$ on $Y_{reg}$ to a nonnegative (1,1)-tensor  $\tilde{g}_Y$ on the whole space $Y$ by setting $\tilde{g}_Y|_{y_j}( \cdot, \cdot) =0$ for $j=1, ..., d$.  Of course, $\tilde{g}_Y$ may be discontinuous at $y_1$, ..., $y_d$.
 Define a distance function $d_Y: Y \times Y \rightarrow \mathbb{R}$ by
\begin{equation}
d_Y(y, y') = \inf_{\gamma}  \int_0^1  \sqrt{ g_Y ( \gamma'(s), \gamma'(s))}  ds,
\end{equation}
where the infimum is taken over all piecewise smooth paths $\gamma: [0,1] \rightarrow Y$ with $\gamma(0)=y$, $\gamma(1)=y'$.
\end{definition}

The goal is to show that such a metric space is exactly the Gromov-Hausdorff limit of $(X, g(t))$ as $t\rightarrow 0$ and it is isomorphic to the metric completion of $(Y_{reg}, g_Y)$.

\begin{theorem} \label{proof2} $(Y, d_Y)$ is a compact metric space homeomorphic to the projective variety $Y$ itself. Furthermore, $(X, g(t))$ converges to $(Y, d_Y)$ in Gromov-Hausdorff topology as $t\rightarrow 0$.

\end{theorem}
\begin{proof}  The same argument in section 3 in \cite{SW2} can be applied to prove the proposition with uniform estimates from Proposition \ref{keyest} and Proposition \ref{diam2}.

\end{proof}

\begin{theorem}\label{proof1}  Let $(\tilde Y, d_{\tilde Y} )$ be the metric completion of $(Y_{reg}, g_Y)$. Then $(\tilde Y, d_{\tilde Y} )$ is isomorphic to $(Y, d_Y)$.

\end{theorem}

\begin{proof}   There are two ways to complete the proof of  the theorem.  The first approach is purely analytic.  We can modify the argument in \cite{SW3} to show that  $(\tilde{Y}, d_{\tilde Y} )$ is homeomorphic to  $Y$ as a projective variety and indeed, $(X, g(t))$ converges to $(Y, d_Y)$ in Gromov-Hausdorff topology. The details can also be found in \cite{SW3} and in \cite{S} for higher codimensional surgeries by the K\"ahler-Ricci flow. This approach does not make use of the general theory on Riemannian manifolds with bounded Ricci curvature \cite{C, CC1, CC2, CCT} and it gives explicit estimates to understand the analytic and geometric contractions.

The second approach relies on the results in \cite{RZ}.  By Theorem \ref{RZ},  $(X, g(t))$ converges to $(\tilde Y, d_{\tilde Y} )$ in Gromov-Hausdorff topology, and hence by the uniqueness of the limiting metric space,  $$(Y, d_Y)=(\tilde Y, d_{\tilde Y} ).$$ The theorem follows from Theorem \ref{proof1}.

\end{proof}

\noindent{\it \large Proof of main results.} Theorem \ref{main1} follows immediately from Theorem \ref{proof1} and Theorem \ref{proof2}. Corollary \ref{main2} is a straightforward consequence of Theorem \ref{main1}. Corollary is proved by combining Theorem \ref{main1} and Theorem \ref{RZ}.

\bigskip

\noindent{\it \large Discussions.} We now discuss some generalizations of Theorem \ref{main1} and future questions. First of all, Theorem \ref{main1}, Corollary \ref{main2} and \ref{main3} can be generalized to higher dimensional conifold small contractions, flops and transitions with little modification. The same estimates can be applied to high codimensional surgery by the K\"ahler-Ricci flow if the exceptional locus is $\PP^n$ with normal bundle $\OO_{\PP^n} (-a_1)\oplus ... \oplus \OO_{\PP^n} (-a_{m+1})$ for $a_i \in \mathbb{Z}^+$. Consequentially, it is shown in \cite{S} that the K\"ahler-Ricci flow  performs certain family of  flipping contractions and resolution, in Gromov-Hausdorff topology. The blow-up limit of Type I singularities of the Ricci flow is a complete shrinking Ricci soliton \cite{H1, EMT}. In the case of the K\"ahler-Ricci flow, finite time singularity arises from contraction of special rational curves. It is natural to conjecture that the curvature tensor blows up at the same rate as the extinction rate of such rational curves. We now make the following conjecture.

\begin{conjecture} With the same assumptions in Theorem \ref{main1}, there exists $C>0$ such that for all $t\in (0, 1]$, such that  the curvature tensor $Rm(t)$ of $g(t)$ is bounded as below
\begin{equation}
\sup_X |Rm(t)|_{g(t)} \leq C t^{-1}.
\end{equation}
Furthermore, the rescaled Ricci-flat K\"ahler metrics $\tilde g(t)$ of $g(t)$ converge to the Ricci-flat K\"ahler metric $g_{CY, \hat E}$ on $\hat E$  with an isolated cone singularity given by  (\ref{limcy}), in pointed Gromov-Hausdorff topology.

\end{conjecture}

In particular, the metric singularity of $g_Y$ near the ordinary double point should be asymptotically close to the local model  given by (\ref{limcy}).



\bigskip
\bigskip

\noindent{\bf Acknowledgements} The author would like to thank Yuguang Zhang for a number of  enlightening discussions and Valentino Tosatti for many helpful suggestions.

\end{document}